\documentclass[11pt,review]{article}
\usepackage{amsmath,amsthm,amssymb,amscd}
\usepackage{graphicx}
\usepackage{graphics}
\usepackage{verbatim}

\usepackage{mathrsfs}
\usepackage{color}
\usepackage[top=1in, bottom=1in, left=1.25in, right=1.25in]{geometry}
\usepackage{enumerate}

\newtheorem{theorem}{Theorem}[section]
\newtheorem{proposition}[theorem]{Proposition}

\newtheorem{lemma}[theorem]{Lemma}

\begin{document}
	
	\title{The chromatic spectrum of signed graphs}
	
	\vspace{3cm}
	
	\author{Yingli Kang \thanks{Fellow of the International Graduate School ``Dynamic Intelligent Systems''; yingli@mail.upb.de}, Eckhard Steffen\thanks{
			Paderborn Institute for Advanced Studies in
			Computer Science and Engineering,
			Paderborn University,
			Warburger Str. 100,
			33098 Paderborn,
			Germany;			
			es@upb.de}}
	
	\date{}
	
	\maketitle

	\begin{abstract}
The chromatic number $\chi((G,\sigma))$ of a signed graph $(G,\sigma)$ is the smallest number $k$ for which there is a function 
$c : V(G) \rightarrow \mathbb{Z}_k$ such that $c(v) \not= \sigma(e) c(w)$ for every edge $e = vw$. 
Let $\Sigma(G)$ be the set of all signatures of $G$. We study the 
chromatic spectrum  $\Sigma_{\chi}(G) =  \{\chi((G,\sigma))\colon\ \sigma \in \Sigma(G)\}$ of $(G,\sigma)$. 
Let $M_{\chi}(G) = \max\{\chi((G,\sigma))\colon\ \sigma \in \Sigma(G)\}$, and $m_{\chi}(G) = \min\{\chi((G,\sigma))\colon\ \sigma \in \Sigma(G)\}$.
We show that $\Sigma_{\chi}(G) = \{k : m_{\chi}(G) \leq k \leq M_{\chi}(G)\}$. We also prove some basic facts for critical graphs.

Analogous results are obtained for a notion of vertex-coloring of signed graphs which was introduced by 
M\'a\v cajov\'a, Raspaud, and \v Skoviera in \cite{MRS_2014}.

	\end{abstract}

\section{Introduction}
Graphs in this paper are simple and finite.  The vertex set of a graph $G$ is denoted by $V(G)$, and the edge set by $E(G)$.
A signed graph $(G,\sigma)$ is a graph $G$ and a function $\sigma : E(G) \rightarrow \{ \pm 1 \}$,
which is called a signature of $G$. The set $N_{\sigma} = \{e : \sigma(e) = -1\}$ is the set of negative edges of $(G,\sigma)$ and
$E(G) - N_{\sigma}$ the set of positive edges. For $v \in V(G)$, let $E(v)$ be the set of edges which are incident to $v$.
A switching at $v$ defines a graph $(G,\sigma')$ with $\sigma'(e) = -\sigma(e)$ for $e \in E(v)$ and $\sigma'(e) = \sigma(e)$ otherwise.
Two signed graphs $(G,\sigma)$ and $(G,\sigma^*)$ are {\em equivalent} if they can be obtained from each other by a sequence of switchings. We also say that
$\sigma$ and $\sigma^*$ are equivalent signatures of $G$.

A circuit in $(G,\sigma)$ is balanced, if it contains an even number of negative edges; otherwise it is unbalanced. The graph
$(G,\sigma)$ is unbalanced, if it contains an unbalanced circuit; otherwise $(G,\sigma)$ is balanced. 
It is well known (see e.g.~\cite{Raspaud_Zhu_2011})
that $(G,\sigma)$ is balanced if and only if it is equivalent to the signed graph with no negative edges, and $(G,\sigma)$ is antibalanced if
it is equivalent to the signed graph with no positive edges. 
Note, that a balanced bipartite graph is also antibalanced. The underlying unsigned graph of $(G,\sigma)$ is denoted by $G$.

In the 1980s Zaslavsky \cite{Zaslavsky_1982, Zaslavsky_1982_2, Zaslavsky_1984} started studying vertex colorings of signed graphs.
The natural constraints for a coloring $c$ of a signed graph $(G,\sigma)$ are, that (1) $c(v) \not= \sigma(e) c(w)$ for each edge $e=vw$,
and (2) that the colors can be inverted under switching, i.e.~equivalent signed graphs have the same chromatic number.
In order to guarantee these properties of a coloring, Zaslavsky \cite{Zaslavsky_1982} used the set $\{-k, \dots, 0, \dots ,k\}$ of $2k+1$ "signed colors" and studied the 
interplay between colorings and zero-free colorings through the chromatic polynomial. 

Recently, M\'a\v{c}ajov\'a, Raspaud, and \v{S}koviera \cite{MRS_2014}
modified this approach. If $n = 2k+1$, then let $M_n = \{0, \pm 1, \dots,\pm k\}$, and
if  $n = 2k$, then let $M_n = \{\pm 1, \dots,\pm k\}$.  
A mapping $c$ from $V(G)$ to $M_n$ is a {\em signed $n$-coloring} of $(G,\sigma)$, if $c(v) \not= \sigma(e) c(w)$ for each edge $e=vw$.
They  define $\chi_{\pm}((G,\sigma))$ to be the smallest number $n$ such that $(G,\sigma)$ has a signed $n$-coloring. 
We also say that $(G,\sigma)$ is {\em signed $n$-chromatic}.

In \cite{KS_2015} we study circular coloring of signed graphs. The related integer $k$-coloring of a signed graph $(G,\sigma)$ is defined as follows.
Let $\mathbb{Z}_k$ denote the cyclic group of integers modulo $k$, and the inverse of an element $x$ is denoted by $-x$. 
A function $c : V(G) \rightarrow \mathbb{Z}_k$ is a $k$-coloring of $(G,\sigma)$ if $c(v) \not= \sigma(e) c(w)$ for each edge $e=vw$. Clearly,
such colorings satisfy the constrains (1) and (2) of a vertex coloring of signed graphs. 
The {\em chromatic number} of a signed graph $(G,\sigma)$ is the smallest $k$ such that $(G,\sigma)$ has a $k$-coloring.
We also say that $(G,\sigma)$ is {\em $k$-chromatic}.

The following proposition describes the relation between these two coloring parameters for signed graphs. 

\begin{proposition} [\cite{KS_2015}] \label{pm}
If $(G,\sigma)$ is a signed graph, then $\chi_{\pm}((G,\sigma)) - 1 \leq \chi((G,\sigma)) \leq \chi_{\pm}((G,\sigma)) + 1$.
\end{proposition}

Let $G$ be a graph and $\Sigma(G)$ be the set of pairwise non-equivalent signatures on $G$.

The {\em chromatic spectrum} of $G$ is the set $\{\chi((G,\sigma))\colon\ \sigma \in \Sigma(G)\}$, which is denoted by $\Sigma_{\chi}(G)$.
Analogously,
{\em signed chromatic spectrum} of $G$ is the set $\{\chi_{\pm}((G,\sigma))\colon\ \sigma \in \Sigma(G)\}$. It is denoted by $\Sigma_{\chi_{\pm}}(G)$.
 Let $M_{\chi}(G) = \max\{\chi((G,\sigma))\colon\ \sigma \in \Sigma(G)\}$ and $m_{\chi}(G) = \min\{\chi((G,\sigma))\colon\ \sigma \in \Sigma(G)\}$.
Analogously, $M_{\chi_{\pm}}(G) = \max\{\chi_{\pm}((G,\sigma))\colon\ \sigma \in \Sigma(G)\}$ and $m_{\chi_{\pm}}(G) = \min\{\chi_{\pm}((G,\sigma))\colon\ \sigma \in \Sigma(G)\}$.

The following theorems are our main results. 

\begin{theorem} \label{cs}
If $G$ is a graph, then $\Sigma_{\chi}(G) = \{ k : m_{\chi}(G) \leq k \leq M_{\chi}(G)\}$. 
\end{theorem}

\begin{theorem} \label{cs_pm}
If $G$ is a graph, then $\Sigma_{\chi_{\pm}}(G) = \{k : m_{\chi_{\pm}}(G) \leq k \leq M_{\pm}(G)\}$.
\end{theorem}

Theorems \ref{cs} and \ref{cs_pm} will be proved in Sections \ref{chrom_spec} and \ref{sig_chrom_spec}, respectively.

\section{The chromatic spectrum of a graph}  \label{chrom_spec}

We start with the determination of $m_{\chi}(G)$. 

\begin{proposition} \label{bipartite}
	Let $G$ be a nonempty graph. The following statements hold.
	\begin{enumerate}
		\item $\Sigma_{\chi}(G) = \{1\}$ if and only if $m_{\chi}(G) = 1$ if and only if $E(G)= \emptyset$.
		\item if $E(G) \not = \emptyset$, then 
		$\Sigma_{\chi}(G) = \{2\}$ if and only if $m_{\chi}(G) = 2$ if and only if $G$ is bipartite.
		\item If $G$ is not bipartite, then $m_{\chi}(G) = 3$.
	\end{enumerate}
\end{proposition}

\begin{proof}
Statements 1. and 2. are obvious. For statement 3 consider $(G,\sigma)$ where $\sigma$ is the signature with all edges negative. 
Then $c : V(G) \rightarrow \mathbb{Z}_3$ with $c(v)=1$ is a 3-coloring of $G$. Since
$G$ is not bipartite the statement follows with statements 1. and 2.
\end{proof}

If $(G,\sigma)$ is a signed graph and $u \in V(G)$, then $\sigma_u$ denotes the restriction of $\sigma$ to $G-u$.
A $k$-chromatic graph $(G,\sigma)$ is {\em $k$-chromatic critical} if $\chi((G-u,\sigma_u)) < k$,  for every $u \in V(G)$. 
In the following proposition, we will give some basic facts on $k$-chromatic critical graphs.
The complete graph on $n$ vertices is denoted by $K_n$.

\begin{proposition}\label{basic_p}
Let $(G,\sigma)$ be a signed graph. 	
\begin{enumerate}
	\item $(G,\sigma)$ is 1-critical if and only if $G = K_1$
	\item $(G,\sigma)$ is 2-critical if and only if $G = K_2$. 
	\item $(G,\sigma)$ is 3-critical if and only if $G$ is an odd circuit.
\end{enumerate}
\end{proposition}
\begin{proof}
Statements 1. and 2. are obvious. An odd circuit with any signature is 3-critical. For the other direction let $G$ be a 3-critical graph.
Note, that  (*) $G-u$ is bipartite for every $u \in V(G)$ by Lemma \ref{bipartite}. 
Since $G$ is not bipartite it follows that every vertex of $G$ is contained in all odd circuits of $G$,
and by (*) every odd circuit $C$ is hamiltonian. $C$ cannot contain a chord, since for otherwise $G$ contains a non-hamiltonian odd circuit,
a contradiction.
Hence, $G$ is an odd circuit.

\end{proof}

\begin{lemma}\label{lem_remove'}
Let $k \geq 1$ be an integer. If $(G,\sigma)$ is $k$-chromatic, then $\chi((G-u,\sigma_u))\in \{k,k-1\}$, for every $u \in V(G)$. In particular,
if $(G,\sigma)$ is $k$-critical, then $\chi((G-u,\sigma_u)) = k-1$.
\end{lemma}

\begin{proof}
For $k\in \{1,2\}$, the statement follows with Proposition \ref{bipartite}. Hence, we may assume that $k\geq3$.
Clearly, $\chi((G-u,\sigma_u)) \leq \chi((G,\sigma)) = k$.
Suppose to the contrary that $\chi((G-u,\sigma_u))\leq k-2$, and 
let $\phi$ be a $(k-2)$-coloring of $(G-u,\sigma_u)$.
We extend $\phi$ to a $(k-1)$-coloring of $(G,\sigma)$.
If $k$ is odd, then change color $x$ to $x+1$ for each $x\geq \frac{k-1}{2}$ and assign color $\frac{k-1}{2}$ to vertex $u$, and we are done.
If $k$ is even, then change color $x$ to $x+1$ for each $x\geq \frac{k}{2}$, and assign color $\frac{k}{2}$ to vertex $u$.
If $\phi(v)=\frac{k-2}{2}$ for a vertex $v$ and $\sigma(uv)=-1$, then recolor $v$ with color $\frac{k}{2}$ to obtain
a $(k-1)$-coloring of $(G,\sigma)$. Hence $\chi((G,\sigma)) \leq k-1 < k$, a contradiction. 
Clearly, if $(G,\sigma)$ is $k$-critical, then $\chi((G-u,\sigma_u)) = k-1$.
\end{proof}

The following statements are direct consequence of Lemma \ref{lem_remove'}.

\begin{theorem} \label{critical}
Let $(G,\sigma)$ be a signed graph and $k \geq 1$. If $\chi((G,\sigma)) = k$, then
$(G,\sigma)$ contains an induced $i$-critical subgraph for each $i \in \{1, \dots,k\}$.
\end{theorem}

\begin{lemma}\label{lem_subgraph_reachable'}
Let $k \geq 3$ be an integer and $H$ be an induced subgraph of a graph $G$. If $k \in {\Sigma_{\chi}(H)}$, then $k\in{\Sigma_{\chi}(G)}$.
\end{lemma}

\begin{proof}
If $k\in{\Sigma_{\chi}(H)}$, then there is a signature $\sigma$ of $H$ such that $\chi((H,\sigma))=k$. Let
$\phi$ be a $k$-coloring of $(H,\sigma)$. Define a signature $\sigma'$ of $G$ as follows. Let $e \in E(G)$ with $e=uv$.

If $e\in E(H)$, then $\sigma'(e)=\sigma(e)$,

if $u,v\notin V(H)$ or if  $u\in V(H), v\notin V(H)$ and $\phi(u)=1$, then $\sigma'(e)=-1$,

if $u\in V(H), v\notin V(H)$ and $\phi(u)\neq 1$, then $\sigma'(e)=1$.

	It follows that $\phi$ can be extended to a $k$-coloring of $(G,\sigma')$ by assigning color 1 to each vertex of $V(G)\setminus V(H)$.
	Thus $\chi((G,\sigma'))\leq k$.
	Moreover, $(G,\sigma')$ has $(H,\sigma)$ as a subgraph with chromatic number $k$, hence, $\chi((G,\sigma'))\geq k$.
	Therefore, $\chi((G,\sigma'))=k$ and thus, $k\in{\Sigma_{\chi}(G)}$.
\end{proof}

\begin{theorem} \label{thm_reachable'}
Let $k \geq 4$ be an integer and $G$ be a graph. If $k\in{\Sigma_{\chi}(G)}$, then $k-1\in{\Sigma_{\chi}(G)}$.
\end{theorem}

\begin{proof}
By Theorem \ref{critical}, $(G,\sigma)$ contains an induced $k$-critical subgraph $(H,\sigma')$, where $\sigma'$ is the restriction of
$\sigma$ to $H$. Since $k \geq 4$, it follows that $|V(H)| > 3$. Hence, there is $u \in V(H)$ such that $\chi(H-u,\sigma'_u)= k-1$.	
Furthermore, $H-u$ is an induced subgraph of $G$. 
Thus,  $k-1\in{\Sigma_{\chi}(H-u)}$, and hence, $k-1\in{\Sigma_{\chi}(G)}$ by Lemma \ref{lem_subgraph_reachable'}.

Note that if $k=3$, then by Proposition \ref{bipartite}, $G$ is not a bipartite graph and thus $k$ can not be decreased to 2.
\end{proof}

Theorem \ref{cs} follows from Proposition \ref{bipartite} and Theorem \ref{thm_reachable'}.

\section{The signed chromatic spectrum of a graph} \label{sig_chrom_spec}

\begin{proposition} \label{char_pm_color}
Let $G$ be a nonempty graph. The following statements hold.
\begin{enumerate}
\item $\Sigma_{\chi_{\pm}}(G) = \{1\}$ if and only if $E(G)= \emptyset$.
\item if $E(G) \not = \emptyset$, then $m_{\chi_{\pm}}(G) = 2$.
\end{enumerate}
\end{proposition}

\begin{proof}
Statement 1. is obvious. For statement 2, since $G$ has at least one edge it can not be colored by using only one color,  hence $m_{\chi_{\pm}}(G) = 2$.
\end{proof}

A signed $k$-chromatic graph $(G,\sigma)$ is {\em signed $k$-chromatic critical} if $\chi_{\pm}((G-u,\sigma_u)) < k$,  for every $u \in V(G)$. 
In  \cite{Stiebitz_2015} Schweser and Stiebitz defined a graph $(G,\sigma)$ to be critical with respect to $\chi_{\pm}$ if
$\chi_{\pm}((H,\sigma')) < \chi_{\pm}((G,\sigma))$ for every proper signed subgraph $(H,\sigma')$ of $(G,\sigma)$, where $\sigma'$ is the restriction of $\sigma$ to $E(H)$.
However, for trees and circuits the two definitions coincide.
The analog statement to Proposition \ref{basic_p} for signed colorings is due to Schweser and Stiebitz in \cite{Stiebitz_2015}.

\begin{proposition}  [\cite{Stiebitz_2015}]  \label{basic_p'}
Let $(G,\sigma)$ be a signed graph. 	
\begin{enumerate}
		\item $(G,\sigma)$ is signed 1-critical if and only if $G = K_1$
		\item $(G,\sigma)$ is signed 2-critical if and only if $G = K_2$. 
		\item $(G,\sigma)$ is signed 3-critical if and only if $G$ is a balanced odd circuit or an unbalanced even circuit.  
	\end{enumerate}
\end{proposition}

\begin{lemma}\label{lem_remove}
Let $k \geq 1$ be an integer. If $(G,\sigma)$ is signed $k$-chromatic, then $\chi_{\pm}((G-u,\sigma_u))\in \{k,k-1\}$, for every $u \in V(G)$. In particular,
if $(G,\sigma)$ is signed $k$-critical, then $\chi_{\pm}((G-u,\sigma_u)) = k-1$.
\end{lemma}

\begin{proof}
For $k\in \{1,2\}$, the statement follows with Proposition \ref{char_pm_color}. Hence, we may assume that $k\geq3$.
Clearly, $\chi_{\pm}((G-u,\sigma_u)) \leq \chi_{\pm}((G,\sigma)) = k$.
Suppose to the contrary that $\chi_{\pm}((G-u,\sigma_u))\leq k-2$ and let $\phi$ be a
$(k-2)$-coloring of $(G-u,\sigma_u)$.
We shall extend $\phi$ to a $(k-1)$-coloring of $(G,\sigma)$.
If $k$ is even, then assign color 0 to vertex $u$, we are done.
If $k$ is odd, then assign color $\frac{k-1}{2}$ to vertex $u$, and for each vertex $v$ such that $\phi(v)=0$ and $\sigma(uv)=-1$, recolor $v$ with color $\frac{k-1}{2}$, and for each vertex $v$ such that $\phi(v)=0$ and $\sigma(uv)=1$, recolor $v$ with color $-\frac{k-1}{2}$ to obtain a $(k-1)$-coloring of $(G,\sigma)$.
Hence $\chi_{\pm}((G,\sigma))\leq k-1<k$, a contradiction. Clearly, if $(G,\sigma)$ is signed $k$-critical, then $\chi_{\pm}((G-u,\sigma_u))=k-1$.
\end{proof}

\begin{theorem} \label{s_critical}
	Let $(G,\sigma)$ be a signed graph and $k \geq 1$. If $\chi_{\pm}((G,\sigma)) = k$, then
	$(G,\sigma)$ contains an induced signed $i$-critical subgraph for each $i \in \{1, \dots,k\}$.
\end{theorem}

\begin{lemma}\label{lem_subgraph_reachable}
	Let $k\geq 2$ be an integer and $H$ be an induced subgraph of a graph $G$. If $k\in \Sigma_{\chi_{\pm}}(H)$, then $k\in \Sigma_{\chi_{\pm}}(G)$. 
\end{lemma}
The proof is similar to the proof of Lemma \ref{lem_subgraph_reachable'}.

\begin{theorem} \label{thm_reachable}
Let $k\geq3$ be an integer and $G$ be a graph. If $k\in \Sigma_{\chi_{\pm}}(G)$, then $k-1\in \Sigma_{\chi_{\pm}}(G)$.
\end{theorem}

\begin{proof}
	
By Theorem \ref{s_critical}, $(G,\sigma)$ contains an induced signed $k$-critical subgraph $(H,\sigma')$, where $\sigma'$ is the restriction of
$\sigma$ to $H$. Since $k \geq 3$, it follows that $|V(H)| \geq 3$. Hence, there is $u \in V(H)$ such that $\chi_{\pm}(H-u,\sigma'_u)= k-1$.	
Furthermore, $H-u$ is an induced subgraph of $G$. 
Thus,  $k-1\in{\Sigma_{\chi_{\pm}}(H-u)}$, and hence, $k-1\in{\Sigma_{\chi_{\pm}}(G)}$ by Lemma \ref{lem_subgraph_reachable}.	
\end{proof}

Theorem \ref{cs_pm} follows from Proposition \ref{char_pm_color} and Theorem \ref{thm_reachable}.




\end{document}